\documentclass[12pt]{amsart}
\usepackage{amsmath}
\usepackage{amsfonts}
\usepackage{amssymb} 
\usepackage{graphicx} 

\usepackage[normalem]{ulem}
\newcounter{corr}
\newcommand{\corr}[3]{\typeout{Warning : a correction remains in page
\thepage}
				\stepcounter{corr}        
				{\color{blue}\ifmmode\text{\,\sout{\ensuremath{#1}}\,}\else\sout{#1}\fi}
       {\color{red}#2}
       {\color{violet} #3}}
\numberwithin{equation}{section}
\def\ran{\mathop{\rm Ran}\nolimits}
 \def\dom{\mathop{\rm dom}\nolimits}  
 \newtheorem{thm}{Theorem}[section]
 \newtheorem{lem}[thm]{Lemma}
 \newtheorem{exam}[thm]{Example}
 \newtheorem{prop}[thm]{Proposition}
 \newtheorem{cor}[thm]{Corollary}
 \newtheorem{rem}[thm]{Remark}
 \newtheorem{defn}[thm]{Definition}
 \newtheorem{assu}[thm]{Assumption}

\newcommand{\Tau}{{Z}}

 \def\Id{\mathop{\rm Id}\nolimits}
 \def\re{\mathop{\rm Re}\nolimits}

 \def\dom{\mathop{\rm dom}\nolimits}  
 \def\r{R}

\newcommand{\C}{{\mathbb C}}   
 \newcommand{\R}{{\mathbb R}}           
\newcommand{\K}{{\mathbb K}}        
   
\renewcommand{\d}{{\mathcal D}}

\author{W. Arendt}
\address{Wolfgang Arendt, Institute of Applied Analysis, University of Ulm. Helmholtzstr. 18, D-89069 Ulm (Germany)} 
\email{wolfgang.arendt@uni-ulm.de}

\author{I. Chalendar}
\address{Isabelle Chalendar,  Universit\'e Gustave Eiffel, LAMA, (UMR 8050), UPEM, UPEC, CNRS, F-77454, Marne-la-Vallée (France)}
\email{isabelle.chalendar@univ-eiffel.fr}

\author{R. Eymard}
\address{Robert Eymard,  Universit\'e Gustave Eiffel, LAMA, (UMR 8050), UPEM, UPEC, CNRS, F-77454, Marne-la-Vallée (France)}
\email{robert.eymard@univ-eiffel.fr}
\title[Initial value problems ]{Lions' representation theorem and applications}
\keywords{Lions' representation theorem, non-autonomous evolution equations, boundary conditions, dissipative operators}
\subjclass[2010]{47A07,47A50,47A52}
\begin{document}	

\begin{abstract}   
 The Representation Theorem of Lions (RTL) is a version of the Lax--Milgram Theorem where completeness of 
 one of the spaces is not complete. In this paper, RTL is deduced from an operator-theoretical version on normed space. The main point of the paper is a theory of derivations, based on RTL, for which well-posedness is proved.  
  One application concerns non-autonomous evolution equations with a new initial-value and a periodic boundary condition for the time variable.
  \end{abstract}	
\maketitle
   
\section{Introduction}
A most elegant way to prove well-posedness of a non-autonomous evolutionary problem has been given by J. L. Lions \cite[Chap. III, Th\'eor\`eme 1.1, p. 37]{Lio61} using a surprising version of the Lax--Milgram Theorem, where one of the spaces is not complete (see e.g. Theorem~\ref{th:1.1real} below). We call it the Representation Theorem of Lions, RTL. The purpose of this  article is to study a derivation problem with the help of Lions' result. A Hilbert space $V$ is given, which we choose real for this introduction, as well as a derivation ${\mathcal D}:R\to V^*$, where $R$ is a dense subspace of $V$;  i.e.  $\mathcal D$ satisfies 
\[    \langle  {\mathcal D}r_1,r_2 \rangle_{V^*, V}+  \langle  {\mathcal D}r_2,r_1 \rangle_{V^*, V}=0  \mbox{ for all }r_1,r_2\in R.\]     
It turns out that this simple setting leads to an interesting structure and to interesting results on boundary conditions. In fact, $\mathcal D$ can be extended naturally to a Hilbert space $W$ which embeds continuously into $V$.   Given a coercive operator ${\mathcal A}\in {\mathcal L}(V,V^*)$,
 we study the problem 
 \begin{equation}\label{eq:fond}
 {\mathcal D}u + {\mathcal A}u=f
 \end{equation}
where $f\in V^*$ is given and $u\in W$ is a solution to be found. We obtain well-posedness results if we impose boundary conditions which are expressed in terms of (strongly) admissible subspaces $\Tau$ of $W$ (see Section~\ref{sec:3} and Section~\ref{sec:4}).  Our main application concerns the evolution equations also considered by Lions (and which have seen a remarkable revival recently; see e.g. Ouhabaz and Spina \cite{OS10}, Haak and Ouhabaz \cite{HO15}, Ouhabaz \cite{Ou15}, Auscher and Egert \cite{AE16}, Dier and Zacher \cite{DZ17}, Achache and Ouhabaz \cite{AO18,AO19} as well as \cite{ADLO14, AM16}). For these
 evolution equations, we obtain a new result whose originality lies in the very  general boundary conditions.   

It looks as follows. Let $U,H$ be Hilbert spaces where $U$ is continuously and densely embedded in $H$ so that we have the Gelfand triple 
 \[   {U}\stackrel{d}{\hookrightarrow} H\hookrightarrow {U}'.\]
Let $W=H^1(0,T;{U}')\cap L^2(0,T;{U})$. Then, as is well-known, $W\subset {\mathcal C}([0,T],H)$. Let $V=L^2(0,T;{U})$ and let ${\mathcal A}\in {\mathcal L}(V,V^*)$ be coercive. Typically, one chooses for ${\mathcal A}$ a non-autonomous differential operator. The following is proved in Section~\ref{sec:appli}. 
\begin{thm}\label{th:int1}
Let $\Phi\in {\mathcal L}(H)$ be a contraction, $f\in  L^2(0,T;{U}')$, $y_0\in H$. Then there exists a unique $u\in W$ such that 
\begin{equation}\label{eq:1.2a}
u'+{\mathcal A}u=f \mbox{ in } L^2(0,T;V')
 \end{equation} 
 \begin{equation}\label{eq:1.2b}
 u(0)=\Phi^*u(T) +y_0\mbox{ in }H.
 \end{equation} 
\end{thm}
Theorem~\ref{th:int1} is known if the form $a:V\times V\to \R$, given by $a(u,v)=\langle {\mathcal A}u,v\rangle_{V^*, V}$ is symmetric (see 
\cite[III. Proposition~2.4]{Sho97}) but the estimates given here for  the non-symmetric case are more demanding. 

In Theorem~\ref{th:int1}, not only well-posedness but also the \emph{maximal regularity} of the solution are remarkable: all three terms of the evolution equation $u'$, ${\mathcal A}u$ and $f$ are in the space  $L^2(0,T;{U}')$ (see \cite{ADF17} for more information on such regularity properties).

The theory of derivations as developed in the present article has applications for completely different subjects. If we identify $V$ and $V'$ by the Riesz Theorem, then a densely defined operator $S$ on $V$ is symmetric if and only if $iS$ is  a derivation. It turns out that our results on boundary operators allow also a description of all selfajoint extensions of the symmetric operator $S$. In fact we refine a version of the theory of boundary triplets as known in the literature. The circle of these ideas is presented in     
\cite{RDV3}. \\

We start this article by a review of the Representation Theorem of Lions (RTL). The new point is that we deduce it from a  result on operators which are defined on normed space (incomplete, in general). This allows us to prove RTL also in a non-Hilbert context (in contrast with Lions \cite[  Th\'eor\`eme 1.1]{Lio61} and Showalter \cite[III. Theorem 2.1]{Sho97}).  We give some   applications of the non-Hilbert version to the Poisson equation for measures in Section~\ref{sec:1}.  The structure of the paper is as follows. 
 \tableofcontents
\section{Representation theorems for continuous linear or antilinear forms and dissipative operators}\label{sec:1}
In this section we describe (generalizations of) the representation theorems of Lions and give an application to dissipative operators.    
\subsection{Representation theorems of Lions}
The representation theorems of Lions (abbreviated by RTL) 
 may be  expressed in terms of operators (which seems to be new), or in terms of bilinear forms or  sesquilinear forms. They depend also heavily on the topological properties (completeness, reflexivity) of the involved normed vector spaces over $\K=\R$ or $\C$.  The antidual of a normed space $X$ is denoted by  $X'$,
whereas the dual space is denoted by $X^*$.

 \begin{thm}[Representation Theorem of Lions, real version]\label{th:1.1real}
 	Let $X$ be a reflexive Banach space and $Y$ be a normed space over $\R$. Let $a:X\times Y\to \R$ be bilinear  such that, for each $y\in Y$,  $x\mapsto a(x,y)$ is continuous on $X$. Let $\beta>0$. The following assertions are equivalent:
 	\begin{itemize}
 		\item[(i)] $\sup_{x\in X,\|x\|_X\leq 1} a(x,y) \geq \beta \|y\|_Y$;
 		\item[(ii)] for all $L\in Y^*$, the dual of $Y$, there exists $x_L\in X$ such that 
 		\[	\|x_L\|\leq \frac{1}{\beta}\|L\|_{Y^*}\mbox{ and }\langle L ,y\rangle_{Y^*,Y} =a(x_L,y\rangle,\,\,\, y\in Y  \]
 	\end{itemize}   
 \end{thm} 
Next we consider the complex version. 
 \begin{thm}[Representation Theorem of Lions, complex version]\label{th:1.1}
Let $X$ be a reflexive Banach space and $Y$ be a normed space over $\C$. Let $a:X\times Y\to \C$ be sesquilinear  such that, for each $y\in Y$,  $x\mapsto a(x,y)$ is continuous on $X$. Let $\beta>0$. The following assertions are equivalent:
\begin{itemize}
	\item[(i)] $\sup_{x\in X,\|x\|_X\leq 1}\re (a(x,y))=\sup_{x\in X,\|x\|_X\leq 1}|a(x,y)|\geq \beta \|y\|_Y$;
	\item[(ii)] for all $L\in Y'$, the antidual of $Y$, there exists $x_L\in X$ such that 
\[	\|x_L\|\leq \frac{1}{\beta}\|L\|_{Y'}\mbox{ and }\langle L ,y\rangle_{Y',Y} =a(x_L,y\rangle,\,\,\, y\in Y  \]
\end{itemize}   
\end{thm} 
In Theorem~\ref{th:1.1}, it is remarkable that $Y$ does not need to be complete and  $a$ does not need to be continuous.

This theorem is a consequence of the following result involving operators instead of sesquilinear forms. 
\begin{thm}\label{th:1.2}
Let $E$ and $F$ be normed spaces over $\K$, $T:E\to F$ linear and $\beta>0$. 
The following assertions are equivalent:
\begin{itemize}
	\item[(i)] $\|Tx\|_F\geq \beta \|x\|_E$ for all $x\in E$;
	\item[(ii)] for all $x^*\in E^*$, there exists $y^*\in F^*$ such that 
	\[ x^*=y^*\circ T\mbox{ and }\|y^*\|\leq \frac{\|  x^*\|}{\beta}.    \]
	\item[(iii)] for all $x'\in E'$, there exists $y'\in F'$ such that 
	\[ x'=y'\circ T\mbox{ and }\|y'\|\leq \frac{\|  x'\|}{\beta}.    \]
\end{itemize}
\end{thm} 
Note that, in Theorem~\ref{th:1.2}, $T$ is not necessarily continuous. On the other side, if $T\in {\mathcal L}(E,F)$, then $x^*=T^* y^*$ and (ii) (resp. (iii)) means that $T^*$ (resp. $T'$) is surjective and $\|T^*\|\leq \frac{1}{\beta}$ (resp. $\|T'\|\leq \frac{1}{\beta}$).   
\begin{proof}
(i)$\Rightarrow$ (ii): Let $Z=TE\subset F$.  Then $T:E\to Z$ is bijective and 
\[   \|T^{-1}z\|_E\leq \frac{1}{\beta}\|z\|_F,\,\, z\in Z.\]	
Let $x^*\in E^*$. Then $x^*\circ T^{-1}\in Z^*$. From the Hahn-Banach theorem, there exists $y^*\in F^*$ such that $y^*_{|Z}=x^*\circ T^{-1}$ and 
\[  \|y^*\|_{F^*}=\|x^*\circ T^{-1}\|_{Z^*}\leq \frac{1}{\beta}\|x^*\|_{E^*}.  \]
Now, for $x\in E$, 
\[   \langle x^*,x\rangle_{E^*,E} = \langle x^* , T^{-1} T x\rangle_{E^*,E} =\langle y^*, Tx\rangle_{Y^*,Y}.   \] 
(ii)$\Rightarrow$(i): Let $x\in E$. Then, by the Hahn-Banach theorem, there exists $x^*\in E^*$, $\|x^*\|\leq 1$ with $\langle x^*, x\rangle_{X^*,X} =\|x\|$. By (ii) there exists $y^*\in F^*$ with $\|y^*\|\leq \frac{1}{\beta}$ and $x^*=y^*\circ T$. It follows that 
\[    \|Tx\|_F\geq \beta \langle y^*, Tx\rangle_{Y^*,Y} =\beta \langle x^*,x\rangle_{X^*,X} =\beta \|x\|_E. \]   
The equivalence with (iii) follows from similar arguments.  
\end{proof}	
\begin{proof}[Proof of Theorem~\ref{th:1.1}]
First note that
 \[\sup_{x\in X,\|x\|_X\leq 1}\re a(x,y) = \sup_{x\in X,\|x\|_X\leq 1} |a(x,y)|.\]
  Indeed, for each $x\in X$ and $y\in Y$, there exists $\theta\in\R$ such that 
  $a(x,y)=e^{i\theta} |a(x,y)|$.   Then $\re a(e^{-i\theta}x,y)=|a(x,y)|$ with $\|e^{-i\theta}x\|=\|x\|$. 
 	
Let $E=Y$ and $F=X'$, the antidual of $X$. Define $T:E\to F$ by $Tw=\overline{a(.,w) }$. Then $T$ is linear and 
\[  \sup_{x\in X,\|x\|_X\leq 1}\re {a(x,y)}\geq \beta \|y\|_Y\Longleftrightarrow \|Ty\|_{X'}\geq \beta \|y\|_Y,   \]
since $\re \overline{a(x,y)}= \re a(x,y)$. 	
By Theorem~\ref{th:1.2}, Condition (ii) of Theorem~\ref{th:1.1} holds if and only if, for all $L\in Y'$, there exists $x_L\in X''\simeq X$ ($X$ is reflexive) such that $\|x_L\|\leq \frac{\|L\|_{Y'}}{\beta}$ and 
\[ \langle L, y\rangle_{Y',Y}  =\langle Ty, x_L\rangle_{X',X}=a(x_L,y),\,\,\, y\in Y  .\]  
\end{proof}

\begin{cor}\label{cor:2.3}
Let $\K=\R$ or $\C$. 	
Let $X$ be a reflexive Banach space and $Y$ be a normed space such that $Y$ embeds continuously in $X$. 
Let $\beta >0$ and let ${a}$ be a sesquilinear form on $X\times Y$ such that
\begin{itemize}
\item[a)] $ |{a}(y,y)|\geq \beta \|y\|_{Y}^2$ for all 
$y\in Y$;
\item[b)] for all $y\in Y$, $x\mapsto {a}(x,y)$ is continuous on $X$.  
 \end{itemize}
Let $L\in Y'$. 
 Then there exists $x_L\in X$ such  that  
\[ 
a(x_L,y)=\langle L, y\rangle_{Y',Y}\mbox{ for all }y\in Y.  \]
\end{cor}
\begin{proof}
	There exists $c_Y>0$ such that $\|y\|_X\leq c_Y\|y\|_Y$ for all $y\in Y$. Thus for $y\in Y\setminus\{0\}$, 
\begin{eqnarray*}
\sup_{x\in X,\|x\|_X\leq 1}|a(x,y)| & \geq & \frac{1}{\|y\|_X}|a(y,y)|\\
  & \geq & \frac{\beta}{\|y\|_X}\|y\|_Y^2\\
   & \geq & \frac{\beta}{c_Y}\|y\|_Y.
\end{eqnarray*}
Now the claim follows from Theorem~\ref{th:1.1}. 
\end{proof}

Lions proved Theorem~\ref{th:1.1real} \cite[III, Th\'eor\`eme 1.1]{Lio61} (see also \cite[III.2 Theorem~2.1]{Sho97}) considering merely Hilbert  (and pre-Hilbert) spaces. Our proof via Theorem~\ref{th:1.2} seems to be new. 

We refer to the circle of results from Theorem~\ref{th:1.1real} to Corollary~\ref{cor:2.3} as RTL, that is, \emph{representation theorems of Lions}. 
\subsection{Dissipative operators}
The following considerations are somehow a side order in this paper. Its aim is to show that Theorem~\ref{th:1.2} is useful also in a  non-Hilbert space situation. We will give a new characterization of dissipativity (Theorem~\ref{th:2.5new}), a notion which will also be used in   Section~\ref{sec:3}.\\

An \emph{operator} $B$ on a Banach space $X$ is a linear mapping from $\dom(B)$ to $X$, where the \emph{domain} $\dom(B)$ of $B$ is a subspace of $X$. By 
\[ \|x\|_B:=\|x\|_X +\|Bx\|_X \]
we denote the \emph{graph norm} of $B$.  The operator $B$ is called \emph{closed} if its graph 
\[G(B):=\{  (x,Bx):x\in \dom(B)\}\]
is closed in $X\times X$. This means that for $x_n\in \dom(B)$
\begin{equation}\label{eq:1.1}
x_n\to x\mbox{ and  } Bx_n\to y\mbox{ in }X\mbox{ implies } x\in \dom(B) \mbox{ and }Bx=y. 
\end{equation}
In other words, $B$ is closed if and only if $(\dom(B),\|\cdot\|_B)$ is complete. 

An operator $B$ is called \emph{closable} if $(0,y)\in \overline{G(B)}\subset X\times X$ implies $y=0$.   Then there exists a unique operator $\overline{B}$ on $X$, the \emph{closure} of $B$ such that $\overline{G(B)}=G(\overline{B})$. 

We call $B$ \emph{dissipative} if  
  \begin{equation}\label{eq:1.2}
  \|x-tBx\| _X\geq \|x\|_X\mbox{ for all }x\in \dom(B),\,\, t>0.
  \end{equation}
If $X$ is a Hilbert space, this is equivalent to 
 \begin{equation}\label{eq:1.3}
 \re \langle Bx,x\rangle_X \leq 0\mbox{ for all }x\in \dom(B).
 \end{equation}
Returning to  the Banach space context, we say that $B$ is \emph{$m$-dissipative} if $B$ is dissipative and $(\Id - tB)$ is surjective for one (equivalently all) $t>0$. Then $(Id -tB)^{-1}\in{\mathcal L}(X)$ and $\|(\Id - tB)^{-1}\|\leq 1$ for all $t>0$.  A densely-defined dissipative operator is closable and its closure 
$\overline{B}$ is dissipative \cite[Lemma~3.4.4]{ABHN11}. 
\begin{rem}
The operator $B$ generates a contractive $C_0$-semigroup on $X$ if and only if $B$ is $m$-dissipative and has dense domain.  
\end{rem}
If $B$ is \emph{densely defined} (i.e., if $\dom(B)$ is dense in $X$), then the adjoint $B^*$
 of $B$ is defined by its domain $\dom(B^*)$ equal to the set
\[  \{  y^*\in X^*:\exists x^*\in X^*\mbox{ such that }\langle y^*,Bx\rangle_{X^*,X} =\langle x^*,x\rangle_{X^*,X},
 	x\in \dom(A) \}\]
 	  and  \[B^*y^*=x^*,\]
 where   $y^*\in \dom(B^*)$ and $x^*\in X^*$ such that $\langle y^*, Bx\rangle_{X^*,X} =\langle x^*,x\rangle_{X^*,X}$ for all $x\in \dom(B)$. Note that $x^*$ is unique since $\dom(B)$ is dense in $X$. 
 
 Theorem~\ref{th:1.2} allows us to give the following dual characterization of dissipativity, which seems to be new. 
 \begin{thm}\label{th:2.5new}
 Let $B$ be a densely defined operator on a Banach space $X$. The following assertions are equivalent.
 \begin{itemize}
 	\item[(i)] $B$ is dissipative; 
 	\item[(ii)] for all $t>0$ and for all $y^*\in X^*$ there exists $x^*\in \dom(B^*)$ such that 
 	\[ x^* -tB^* x^*=y^*\mbox{ and }\|x^*\|_{X^*}\leq \|y^*\|_{X^*}.  \]
 \end{itemize} 
 \end{thm} 
\begin{proof}
Consider the spaces $E=\dom(B)$ with norm $\|x\|_E=\|x\|_B$ and $F=X$ with $\|x\|_F=\|x\|_X$. Let $t>0$ and consider the linear mapping $T:E\to F$ given by $Tx=x-tBx$. Note that $T$ is not continuous unless $B$ is continuous.  Since $\dom(B)$ is dense in $X$, one has $E^*=X^*$. Theorem~\ref{th:1.2} tells us that 
$\|x-tBx\|_X=\| Tx\|_X\geq \|x\|_X$ for all $x\in E$ if and only if for all $x^*\in X^*$ there exists $y^*\in X^*$ such that $\|y^*\|\leq \|x^*\|$ and $x^*=y^*\circ T$, i.e. $\langle x^*,x\rangle_{X^*,X} =\langle y^*,x-tBx\rangle_{X^*,X}$, or equivalently
\[ \langle y^*, tBx\rangle_{X^*,X} =\langle y^*-x^*,x\rangle_{X^*,X}, \,x\in \dom(B).    \] 	
This in turn is equivalent to $y^*\in \dom(B^*)$ and $tB^*y^*=y^*-x^*$, i.e. $y^*-tB^*y^*=x^*$. 
\end{proof}	
A densely defined operator $B$ is $m$-dissipative if and only if $B$ is closed, dissipative and $\Id -tB^*$ is injective for some (equivalently all) $t>0$. This is true since $\Id -tB$ has dense range if and only if $(\Id -tB^*)$ is injective. We illustrate Theorem~\ref{th:2.5new} by the following example. 
\begin{exam}\label{ex:2.6new}
	Let $\Omega\subset \R^d$ be open and bounded, $X={\mathcal C}_0(\Omega):=\{ u\in {\mathcal C}
	(\overline{\Omega}):u_{| \partial \Omega}=0\}$. Let $Bu=\Delta u$, $\dom(B):=\{u\in {\mathcal C}_0(\Omega) :\Delta u\in {\mathcal C}_0(\Omega)\}$. Then $B$ is dissipative, closed and densely defined. Moreover $B$ is $m$-dissipative if and only if $\Omega$ is Wiener regular (see \cite{AB99}). Note that ${\mathcal C}_0(\Omega)'={\mathcal M}(\Omega):=\{\mu : \mu\mbox{ is a signed measure on }\Omega\}$ with $\|\mu\|$ the total variation of $\mu\in {\mathcal M}(\Omega)$. 
	
	For $\mu,\nu\in {\mathcal M}(\Omega)$, one has 
	\[ \mu\in \dom(B^*) \mbox{ and }B^*\mu=\nu\Longleftrightarrow \int_{\Omega}\Delta ud\mu=\int_{\Omega}ud\nu,  u\in \dom(B).\]
	In particular, $\nu=\Delta \mu$ in the sense of distributions. 
	
	Since $B$ is dissipative, we see from Theorem~\ref{th:2.5new} that for each $t>0$ and each 
	$\nu\in {\mathcal M}(\Omega)$, there exist $\mu\in \dom(B^*)$ such that 
	\begin{equation}\label{eq:1.4new}
	\mu-t\Delta \mu=\nu\mbox{ and }\|\mu\|\leq \|\nu\|.
	\end{equation} 
	Moreover, $\Omega$ is Wiener regular if and only if the solution of \eqref{eq:1.4new} is unique for some (equivalently all) $t>0$. 
\end{exam}  
    
\section{When Lions' Representation Theorem meets derivations}\label{sec:3}
We start introducing the  notion of derivation acting on a Hilbert space. Our aim is to study perturbations of a coercive operator by a derivation, which leads to    define Weak and Strong  Derivation Problems.   
In the remaining of the paper, $V$ is a Hilbert space over $\K=\R$ or $\C$. By $V'$ we denote the antidual of $V$, that is the space of all  continuous  antilinear forms on $V$.  In the case $\K=\R$, $V'=V^*$ where $V^*$ is the dual space.
 \subsection{Derivations} 
The following definition is central for the entire article. 
\begin{defn}\label{def:derivation}
A \emph{derivation} (or \emph{$V$-derivation} if we want to specify the Hilbert space) is a linear mapping 
\[  {\mathcal D}_0 : \r\to V',\]
whose domain $\r$ is a dense subspace  of $V$ satisfying 
\[   \langle    {\mathcal D}_0u,v\rangle_{V',V} +\overline{\langle {\mathcal D}_0v,u\rangle_{V',V}}=0, \]
for all $u,v\in\r$. 
\end{defn}
Here, as previously, we denote by $\langle \, ,\, \rangle_{V',V} $ the duality between $V'$ and $V$, whereas $\langle \, ,\, \rangle_V $ is the scalar product in $V$. 
\begin{rem}
In general we do not want to identify $V'$ and $V$. But if we do so, then ${\mathcal D}_0$ is a derivation if and only if ${\mathcal D}_0$ is \emph{antisymmetric}, i.e. 
\[\langle {\mathcal D}_0u,v\rangle_V=-\overline{\langle {\mathcal D}_0v,u\rangle_V}\mbox{ for all }u,v\in \r .  \]
If $\K=\C$ this is the same as saying that $i{\mathcal D}_0$ is symmetric. 
\end{rem}



In the remain of this entire section,   ${\mathcal D}_0$ is a derivation according to Definition~\ref{def:derivation}. We associate to ${\mathcal D}_0$ a new Hilbert space $W$ in the following way. 
  
\begin{equation}\label{eq:3.1}
W:= \{ v\in V:\exists f\in V' \mbox{ such that }\langle  f,r\rangle_{V',V} +\overline{\langle   {\mathcal D}_0r, v\rangle }_{V',V} =0
\mbox{ for all }r\in \r \}.
\end{equation}
Since $\r$ is dense in $V$, given $v\in W$, there exists at most one $ f\in V'$ satisfying the above requirement. We then denote ${\mathcal D}v:=f$.  

The following is easy to see.   

\begin{lem}\label{lem:3.1}
	\begin{itemize}
		\item[a)] $\r\subset W$ and  ${\mathcal D}r={\mathcal D}_0r$ for all $r\in \r$. 
		\item[b)] $W$ is a Hilbert space endowed with the scalar product
		\[    \langle v,w\rangle_W := \langle v,w\rangle_V +\langle {\mathcal D}v,{\mathcal D}w\rangle_{V'},      \]
		where $V'$ carries the scalar product coming from the Riesz isomorphism  between $V$ and $V'$.    
	\end{itemize}
\end{lem}

It follows from this definition that 
\begin{equation}\label{eq:nul}
 \langle  {\mathcal D}v,r\rangle_{V',V} +\overline{\langle   {\mathcal D}r, v\rangle }_{V',V} =0\mbox{ for all }v\in W, r\in \r .
\end{equation} 
Moreover we set 
\begin{equation}\label{eq:3.2}
b(v,w):=\langle {\mathcal D}v,w\rangle_{V',V} +\overline{\langle {\mathcal D}w,v\rangle}_{V',V}\mbox{ for all  }v,w\in W.
\end{equation}
Then $b:W\times W\to\K$ is sesquilinear and symmetric, i.e. 
\[   b(v,w)=\overline{b(w,v)}\mbox{ for all } v,w\in W.\]
We call $W$ the \emph{extended domain} of ${\mathcal D}_0$, ${\mathcal D}$ the extension of ${\mathcal D}_0$, and $b$ the \emph{associated boundary form}.

If $\r$ is dense in $W$, then $b$ is identically equal to $0$. But in general $b$ is different from $0$ and will be used to define boundary conditions for a Derivation Problem associated with $\mathcal D$.  

Note that $b(w,w)\in\R$ for all $w\in W$. 
\begin{defn}\label{def:bortho}
For $E\subset W$ its \emph{orthogonal with respect to $b$} is defined by 
\[  E^b:=\{ u\in W: b(u,v)=0 \,\,\forall v\in E  \}.\]
\end{defn}
\subsection{Admissible subspaces} 
Another notion of great importance in the sequel is the notion of \emph{admissible space} and  \emph{strongly admissible space}.
\begin{defn}\label{def:3.2}
A subspace $\Tau $ of $W$ is called \emph{admissible}  if 
\begin{itemize}
\item[a)] $\r\subset \Tau $,
\item[b)] $b(w,w)\leq 0$ for all $w\in \Tau $. 	
\end{itemize}	
We say that $\Tau $ is \emph{strongly admissible} if, in addition,  
\begin{itemize}
\item[c)] $b(u,u)\geq 0$ for all  $u\in \Tau^b $. 
\end{itemize} 
\end{defn}
Since $b:W\times W\to\K$ is continuous, the closure of an admissible subspace $Z$ is admissible, and, if $Z$ is strongly admissible, then so is its closure $\overline{Z}$.
\subsection{The Weak Derivation Problem}   
Next we consider the \emph{weak derivation problem}. Let $Z\subset W$ be an admissible and closed subspace. We endow $Z$ with the norm 
\[   \|v\|_Z^2 :=\|v\|_W^2:=\|v\|_V^2 +  \|{\mathcal D}v\|^2_{V'}.\] 
Let ${\mathcal A}\in {\mathcal L}(V,V')$ be \emph{coercive}, i.e. there exists $\alpha>0$ such that  
\[  \langle  {\mathcal A}u,u\rangle_{V',V}\geq \alpha \|u\|_V^2 \mbox{ for all }u\in V.  \]
  Then the weak derivation problem is the following: given $L\in Z'$ find a vector $v$ such that  
 \[ \mbox{(WDP) $v\in V$ and $-\overline{  \langle {\mathcal D}z,v \rangle_{V',V}}  + \langle {\mathcal A}v,z\rangle_{V',V}=\langle L,z\rangle_{Z',Z},\,\,\, z\in Z$.}   \]
\begin{thm}\label{th:3.5new}
Let $Z\subset W$ be a closed admissible subspace and let $L\in Z'$. Then there exists a vector $u$  such that (WDP) holds. 
If  $Z$ is strongly admissible, then the solution is unique.  	
\end{thm}
We consider $\Tau$ as a space of test functions. The solutions of 
$(WDP)$ depend on the choice of this space. We will make this 
dependance more transparent in Theorem~\ref{th:3.9}, Proposition~\ref{prop:5.3n} and Theorem~\ref{th:4.12} below.\\
  
For the proof we need some preparation.
\begin{lem}\label{lem:3.6new}
Let $X$ be a Banach space, $B$ a closed operator on $X$ and $A\in{\mathcal L}(X)$. Assume that there exists $\alpha >0$ such that 
\begin{equation}\label{eq:3.4new}
\|(A-B)x\|_X\geq \alpha \|x\|_X\mbox{ for all }x\in \dom(B).
\end{equation}
Then there exists a constant $\beta>0$ such that 
 \begin{equation}\label{eq:3.5new}
 \|(A-B)x\|_X\geq \beta (\|x\|_X +\|Bx\|_X)\mbox{ for all }x\in \dom(B).
 \end{equation}

\end{lem}
\begin{proof}
By assumption the space $Z:=\dom(B)$  endowed with the norm $\|x\|_Z:=\|x\|_X + \|Bx\|_X$ is complete. The mapping $T:=A-B:Z\to X$ is continuous. In fact 
\[ \|Tz\|_X\leq \|A\|  \|x\|_X +\|Bx\|_X\leq \max \{\|A\|,1\}(\|x\|_X + \|Bx\|_X).  \]	
Moreover,
\begin{equation}\label{eq:3.6n}
\|Tx\|_X\geq \alpha \|x\|_X\mbox{ for all }x\in Z.
\end{equation}
Thus $T$ is injective and $Y:=TZ$ is closed in $X$. To see this let $z_n\in Z$ such that 
$\lim_{n\to\infty}\|Tz_n -y\|_X=0$. It follows from \eqref{eq:3.6n} that 
\[   \alpha \|z_n-z_m\|_X\leq \|Tz_n-Tz_m\|_X.\]
Thus $x:=\lim_{n\to \infty}z_n$ exists in $X$. Since $A\in {\mathcal L}(X)$, it follows that $\lim_{n\to\infty}Az_n=Ax$. Thus $\lim_{n\to\infty} Bz_n=Ax-y$.  Since $B$ is closed it follows that $x\in \dom(B)=Z$ and $Bx=Ax-y$. Hence $Tx=y$. We have proved that $Y$ is closed in $X$ and hence complete. By the Theorem of the Continuous Inverse, $T^{-1}\in {\mathcal L}(Y,Z)$. Thus 
\begin{eqnarray*}
\|z\|_X + \|Bz\|_X & = & \|z\|_Z =\|T^{-1}Tz\|_Z\\
 & \leq & \|T^{-1}\| \|Tz\|_Z=\|T^{-1}\| \| Az - Bz\|_X.	
\end{eqnarray*}
Choose $\beta =\frac{1}{\|T^{-1}\|}$.   
\end{proof}	
On a Hilbert space $V$ we obtain the following more special version. Recall that an operator $B$ on $V$ is dissipative  if and only if
\[   \re \langle Bv,v  \rangle_V \leq 0\mbox{ for all }v\in \dom(B). \] 
\begin{prop}\label{prop:3.7new}
Let $B$ be a densely defined dissipative operator on a Hilbert space $V$ and let $A\in {\mathcal L}(V)$ be \emph{coercive}, i.e. there exists $\alpha>0$ such that 
\begin{equation}\label{eq:3.7n}
\re \langle Av,v\rangle_V\geq \alpha \|v\|_V^2 \mbox{ for all }v\in V.
\end{equation}  
Then there exists $\beta>0$ such that 
 \begin{equation}\label{eq:3.8n}
 \| Av - Bv\|_V\geq \beta (\|v\|_V +\|Bv\|_V) \mbox{ for all }v\in V.
 \end{equation} 
\end{prop}
\begin{proof}
By \cite[II. Proposition~3.14 (iv)]{EN06}  or \cite[I. Theorem~4.5 (c)]{pazy}, $B$ is closable and the closure $\overline{B}$ of $B$ is dissipative again. Thus we can assume that $B$ is closed. One has for $T=A-B:\dom(B)\to V,$
\[\alpha \|v\|_V^2 \leq \re \langle Av,v\rangle_V\leq  \re \langle Tv,v\rangle_V\leq \|Tv\|_V  \|v\|_V. \]
Thus \eqref{eq:3.4new} is satisfied. Now the claim follows from Lemma~\ref{lem:3.6new}.	
\end{proof}	
Our short proof of Lemma~\ref{lem:3.6new} and Proposition~\ref{prop:3.7new} depends on the Theorem of the Continous Inverse. So we do not have a control of the constant $\beta$. For this reason we give a second (more involved) proof which shows that $\beta$ can be chosen independently of $B$.  
 \begin{proof}[Second proof of Proposition~\ref{prop:3.7new}]
 Consider $A_+:=\frac{A+A^*}{2}$ and $A_-:=\frac{A-A^*}{2}$. Then 
 \[\langle A_+ v,v\rangle_V=\frac{1}{2}\langle Av,v\rangle_V + \frac{1}{2}\langle A^*v,v\rangle_V=\re \langle Av,v\rangle_V ,\,\, v\in V,\] 
 and thus $\langle A_+ v,v\rangle\geq \alpha \|v\|_V^2, \,\, v\in V$. It follows that the selfadjoint operator $A_+$ is positive and invertible. Let $S$ be the (invertible and positive) square root of $A_+$. 
 Since $\| Sv\|_V^2=\langle A_+v,v\rangle_V$, we get 
 \begin{equation}\label{eq:pf1}
 \|Sv\|_V\geq \sqrt{\alpha} \|v\|_V,\,\, v\in V.
 \end{equation}
 Since $\|Sv\|\leq \|A_+\|^{1/2}\|v\|$, we also get 
 \begin{equation}\label{eq:pf2}
  \|S^{-1} v\|   \geq \frac{1}{\|A_+\|^{1/2}} \|v\|,\, \, v\in V.  
 \end{equation} 
  
 In addition, note that 
 \[    \langle A_-v,v\rangle_V=\langle v,A^*_- v\rangle_V =-\langle v,A_-v\rangle_V,\,\, v\in V,\]
 and therefore 
 \begin{equation}\label{eq:pf3}
 \re \langle A_- v,v\rangle_V =0,\,\, v\in V.
 \end{equation}
 Using \eqref{eq:pf1} we first get, for all $v\in V$, 
 \begin{eqnarray*}
  \|(A-B)v\|^2\ & \geq &  \alpha (\|-S^{-1}Bv + S^{-1}A_+ v+ S^{-1}A_-v\|_V^2)  \\
   & = &   \alpha( \|S^{-1}Bv\|_V^2 + \|S^{-1}A_+v\|_V^2 +\|S^{-1}A_-v\|_V^2 )  \\
    &  & - 2\alpha\re \langle S^{-1}Bv, S^{-1}A_+ v\rangle_V + 2\alpha\re \langle S^{-1}A_+v, S^{-1}A_- v\rangle_V \\
      &  & - 2\alpha\re \langle S^{-1}Bv, S^{-1}A_- v\rangle_V .  
 \end{eqnarray*}
Note that 
\[\re \langle S^{-1}A_+v, S^{-1}A_- v\rangle_V =\re \langle S^{-2}A_+v, A_- v\rangle_V = \re \langle v, A_- v\rangle_V =0\mbox{ by }\eqref{eq:pf3}.\]
Moreover
\[ -\re \langle S^{-1}Bv, S^{-1}A_+ v\rangle_V= -\re \langle Bv, S^{-2}A_+ v\rangle_V =-\re \langle Bv,  v\rangle_V\geq 0.\]
In addition, since $S^{-1}A_+ = S$,
\[\|S^{-1}A_+v\|_V^2=\|Sv\|_V^2.\]
It follows that  
\begin{eqnarray*}
	\|(A-B)v\|^2\ & \geq &    \alpha (\|S^{-1}Bv\|_V^2 + \|Sv\|_V^2 +\|S^{-1}A_-v\|_V^2 )  \\
	&  & - 2\alpha \re \langle S^{-1}Bv, S^{-1}A_- v\rangle_V .  
\end{eqnarray*}
Now we use the celebrated Peter--Paul inequality combined with the Cauchy-Schwarz to estimate, for all $\gamma>0$,   
\begin{eqnarray*}
2 \left| \re \langle S^{-1}Bv, S^{-1}A_- v\rangle_V \right|& \leq &  2 \|S^{-1}Bv\|_V \|S^{-1}A_-v\|_V\\
 & \leq & \gamma \|S^{-1}Bv\|_V^2 +\frac{1}{\gamma} \|S^{-1}A_-v\|_V^2.
\end{eqnarray*}
We get then 
\[
	\|(A-B)v\|^2\  \geq    \alpha (1-\gamma)\|S^{-1}Bv\|_V^2 +\alpha (1-1/\gamma )\|S^{-1}A_-v\|_V^2
		+\alpha \|Sv\|_V^2.\]
Let $\gamma=\frac{1}{1+s}$ where $s>0$ be be fixed later.  Then $1-\gamma=\frac{s}{1+s}$ and $1-1/\gamma = -s$.  Then, by \eqref{eq:pf1} and \eqref{eq:pf2}, it follows that 
\[ 	\|(A-B)v\|^2 \geq    \alpha\left(  \frac{s}{1+s}\frac{1}{\|A_+\|}\|Bv\|_V^2 +\alpha\|v\|_V^2 -s\|S^{-1}A_-\|^2 \|v\|_V^2      \right) .    \]
Then we choose $s$ such that $s	\|S^{-1}A_-\|^2=\frac{\alpha}{2}$.  Then we obtain 
\[  	\|(A-B)v\|^2 \geq \beta^2 (\|Bv\|_V^2 +\|v\|_V^2),   \]
where 
$\beta^2	=\min \left\{   \frac{\alpha^2}{2}, \frac{\alpha^2}{\alpha +2\|S^{-1}A_-\|^2}\frac{1}{\|A_+\|}   \right\}$. In other words $\beta$ does not depend on $B$. 
 \end{proof}
\begin{proof}[Proof of Theorem~\ref{th:3.5new}]
By the Riesz Representation Theorem there exists an operator $A^*\in{\mathcal L}(V)$ such that 
\[     \langle v,A^*w\rangle_V=\langle {\mathcal A} v,w\rangle_{V',V}\mbox{ for all }v,w\in V. \]
Similarly there exists $B:Z\to  V$ such that 
\[  \langle Bz,w\rangle_V=\langle {\mathcal D}z,w\rangle_{V',V} \mbox{ for all }z\in Z, w\in V. \]
The operator $B$ is closed since $Z$ is complete for the norm 
\[  \|u\|_Z=\left( \|z\|_V^2 +\|{\mathcal D}z\|_{V'}^2  \right)^{1/2}=\left( \|z\|_V^2 +\|Bz\|_{V}^2  \right)^{1/2}.\] 
Since $Z$ is admissible $B$ is dissipative. The operator $A^*$ is coercive. By Lemma~\ref{lem:3.6new} there exists $\beta>0$ such that
\[   \|A^* z-Bz\|_V\geq \beta (\|z\|_V +\|Bz\|_V).\]	
Thus 

\[\begin{array}{lcl}
	\sup_{\|v\|_V\leq 1} | \overline{\langle -{\mathcal D}z,v\rangle_{V',V} } + \langle {\mathcal A}v,z\rangle_{V',V}|
 & = &   \sup_{\|v\|_V\leq 1} | \langle -{\mathcal D}z,v\rangle_{V',V} + \langle A^*z,v\rangle_{V} |  \\
 & = &   \sup_{\|v\|_V\leq 1} | \langle -Bz + A^*z,v\rangle_{V}   |  \\
 & = & \|-Bz+A^*z\|  \\
 & \geq & \beta (\|z\|_V +\|Bz\|_V)=\beta \|z\|_Z  . 
\end{array}
\]
By RTL (Corollary~\ref{cor:2.3}) there exists $v\in V$ such that 
\[  \langle  L,z\rangle =
-\overline{  \langle {\mathcal D}z,v \rangle_{V',V}}  + \langle {\mathcal A}v,z\rangle_{V',V}\mbox{ for all }z\in Z. \]
For the uniqueness, we now assume that $Z$ is strongly admissible. Let $u\in V$ be the difference of two solutions. Then 
\begin{equation}\label{eq:3.7nn}
\overline{\langle -{\mathcal D}z,u\rangle_{V',V}} + { \langle {\mathcal A}u,z\rangle_{V',V} }=0\mbox{ for all }z\in Z.
\end{equation} 
Taking $z\in R$ we see that $u\in W$ and 
\begin{equation*}
{\mathcal D}u +{\mathcal A}u=0\mbox{ in } V'.
\end{equation*}
In particular, 
\begin{equation}\label{eq:3.8nn}
 \langle {\mathcal D}u,z\rangle_{V',V} +\langle {\mathcal A}u , z\rangle_{V',V}=0\mbox{ for all  }z\in Z.
\end{equation}
Substracting \eqref{eq:3.7nn} from \eqref{eq:3.8nn} we deduce that 
\[   b(u,z)= \langle {\mathcal D}u,z\rangle_{V',V} + \overline{\langle {\mathcal D}z,u\rangle_{V',V}}=0\mbox{ for all }z\in Z.\]
Thus $u\in Z^b$. Since $Z$ is strongly admissible, it follows that 
\[  2\re \langle  {\mathcal D}u,u\rangle_{V',V} =b(u,u)\geq 0.   \]
From \eqref{eq:3.8nn} we deduce that 
\[   0=\re \langle {\mathcal D}u,u\rangle_{V',V} +\re \langle {\mathcal A} u,u\rangle_{V',V}\geq \alpha \|u\|_V^2. \]
Thus $u=0$. 
\end{proof}	
 \subsection{The Strong Derivation Problem}
 Next we consider the \emph{strong derivation problem}. For that we have to assume that the antilinear form $L$ on $Z$ has a special form. 
 \begin{lem}\label{lem:3.8}
 Let $Z$ be  a subspace of $W$ endowed with the norm $\|\cdot \|_Z=\|\cdot\|_W$ such that $R\subset Z$, and let $L\in Z'$. The following assertions are equivalent:
 \begin{itemize}
 	\item[(i)] there exists a constant $c\geq 0$ such that $|\langle L,r\rangle_V|\leq c\|r\|_V$ for all $r\in R$;
 	\item[(ii)] there exist $f\in V'$ and $g\in Z'$ such that $\langle g,r\rangle_V=0$ for all $r\in R$ and $L=f+g$.   
 \end{itemize} 
 \end{lem}
 \begin{proof}
 $(ii)\Rightarrow (i)$ This is obvious.\\
  $(i)\Rightarrow (ii)$ Since $R$ is dense in $V$, there exists $f\in V'$ such that $\langle L,r\rangle_{V',V}=\langle f,r\rangle_{V',V}$ for all $r\in R$. Let $g:= L-f$. Then $g\in Z'$ and $g(r)=0$ for all $r\in R$. 
 \end{proof} 
Let $Z\subset W$ be an admissible subspace and let $L\in Z'$ be of the form $L=f+g$ where $f\in V'$ and $g\in Z'$ such that $g(r)=0$ for all $r\in R$. The \emph{Strong Derivation Problem}  consists in finding  a vector $u$ such that 
  \[   \mbox{ (SDP)  } u\in W , \,\, {\mathcal D}u +{\mathcal A}u =f\mbox{ and }b(u,z)=-\langle g,z\rangle_{Z',Z},\,\, z\in Z.\]
Thus, the functional $g$ expresses a kind of boundary condition. In the examples (Section~\ref{sec:appli}), this will become even more transparent.   
\begin{thm}\label{th:3.9}
Let $Z\subset W$ be an admissible subspace. Let $L=f+g$ with $f\in V'$, $g\in Z'$ such that $\langle g,r\rangle_{Z',Z}=0$ for all $r\in R$. For $u\in V$ the following are equivalent: 
\begin{itemize}
	\item[(i)] $u$ is a solution of (WDP);
	 \item[(ii)] $u\in W$ and $u$ is a solution of (SDP).
\end{itemize}
Consequently, there exists a solution of (SDP). Moreover it is unique if $Z$ is strongly admissible.   
\end{thm}	
\begin{proof}
$(i)\Rightarrow (ii)$ Let $u$ be a solution of (WDP). Then 
\begin{equation}\label{eq:3.9}
-\overline{ \langle {\mathcal D}z,u\rangle_{V',V}} +\langle {\mathcal A}u,z\rangle_{V',V}=\langle f,z\rangle_{V',V} +  \langle g,z\rangle_{Z',Z},\,\, z\in Z.
\end{equation} 
Taking $z\in R$ we deduce that $u\in W$ and 
\begin{equation}\label{eq:3.10n}
{\mathcal D}u +{\mathcal A}u=f \mbox{ in }V'.
\end{equation} 
Evaluating at $z\in Z\subset V$, we obtain 
\[  \langle  {\mathcal D}u,z\rangle_{V',V} +\langle {\mathcal A}u,z\rangle_{V',V}=\langle f,z\rangle_{V',V},\,\, z\in Z.  \]
Substracting \eqref{eq:3.9} from \eqref{eq:3.10n} yields
\begin{equation}\label{eq:3.11}
b(u,z)=\langle {\mathcal D}u,z\rangle_{V',V} +\overline{ \langle {\mathcal D}z,u\rangle_{V',V}  }=-\langle g,z\rangle_{Z',Z},\,\, z\in Z.
\end{equation}	
$(ii)\Rightarrow (i)$ Let $u\in W$ be a solution of (SDP), i.e. \eqref{eq:3.10n} and \eqref{eq:3.11}  hold. Evaluating \eqref{eq:3.10n} at $z\in Z$ and using \eqref{eq:3.11} we obtain 
\begin{eqnarray*}
   -\overline{ \langle	{\mathcal D}z,u\rangle_{V',V}} +\langle {\mathcal A}u,z\rangle_{V',V} 
 & = & -\overline{	\langle {\mathcal D}z,u\rangle_{V',V}} -\langle {\mathcal D}u,z\rangle_{V',V} + 
 \langle f,z\rangle_{V',V}\\
  & = & -b(u,z) + \langle f,z\rangle_{V',V}\\
   & = &\langle g,z\rangle_{Z',Z} + \langle f,z\rangle_{V',V}\\
    & = & \langle L,z\rangle_{V',V},\,\, z\in Z.
\end{eqnarray*}
Thus $u$ is  a solution of (WDP). 
\end{proof}	
\begin{rem}\label{rem:zigoto}  Letting $g=0$, Theorem \ref{th:3.9} proves that the operator ${\mathcal D} +{\mathcal A}~:~Z^b\to V'$ is invertible, i.e.  it is injective with a closed and dense range. This means that there exists $\beta'>0$ such that, for all $u\in Z^b$,
\begin{equation}\label{eq:min}
 \Vert {\mathcal D}u +{\mathcal A}u\Vert_{V'} \ge \beta' \left( \|u \|_V^2 +\| {\mathcal D}u\|_{V'}^2  \right)^{1/2}
\end{equation}
and that, for any $w\in V$, 
\begin{equation}\label{eq:nul}
	 \langle {\mathcal D}v +{\mathcal A}v,w\rangle_{V',V} = 0,\,\,\,v\in Z^b \Rightarrow w=0.
\end{equation}
  Indeed \eqref{eq:min} is equivalent to ${\mathcal D} +{\mathcal A}$ is injective with closed range, whereas  \eqref{eq:nul} is equivalent to  ${\mathcal D} +{\mathcal A}$ has dense range using the Hahn--Banach theorem combined with the reflexivity of our spaces. 

Note that the existence of $\beta'$ follows directly from  Lemma~\ref{lem:3.6new}. Indeed, for all $u\in Z^b$, $b(u,u)\ge 0$, and the second property follows along the same lines as the proof of $(i)\Rightarrow (ii)$ above (first prove that $w\in W$, then that $w\in Z^{bb} = Z$, and finally that $\re \langle{\mathcal A}^\star w,w\rangle_{V',V}\le 0$).
\end{rem}
 
 \section{Admissible subspaces and boundary operators}\label{sec:5n}
If we  consider $\mathcal D$ as a time derivative, then (SDP) may be interpreted  as an evolution equation with a \emph{boundary  condition with respect to time}. 

This can be made more concrete when we make further assumptions on the associated boundary form $b$  and, of course, in the examples we give below.  
\subsection{Maximal admissible subspaces}

We start considering maximal admissible subspaces. This maximality is a handy criterion for proving strong admissibility.   
 \begin{defn}\label{def:3.6}
 A subspace $\Tau_0$ of $W$ is called \emph{maximal admissible} if $\Tau_0$ is admissible and if for any admissible subspace $\Tau_1$ of $W$, $\Tau_0\subset \Tau_1$ implies $\Tau_0=\Tau_1$.     
 \end{defn}   
\begin{prop}\label{prop:3.7}
Let $\Tau_0$ be an admissible subspace of $W$. Then there exists a maximal admissible subspace $\Tau_1$ of $W$ containing $\Tau_0$. Moreover, each maximal admissible subspace of $W$ is strongly admissible.    
\end{prop}	
\begin{proof}
Let ${\mathcal M}:=\{  \Tau_1\subset W:\Tau_0\subset \Tau_1, \Tau_1\mbox{ admissible}\}$ ordered by inclusion. If $(\Tau_i)_{i\in I}$ is a chain in ${\mathcal M}$, then $\Tau :=\cup_{i\in I} \Tau_i$ is a subspace of $W$ containing $\Tau_0$. If $w\in \Tau $, then there exists $i\in I$ such that $w\in \Tau_i$. Consequently $b(w,w)\leq 0$. Thus $\Tau $ is admissible and hence an upper bound for the chain $(\Tau_i)_{i\in I}$. Now by Zorn's Lemma ${\mathcal M}$ has a maximal element.

Let $\Tau_1$ be a maximal admissible  subspace and let $u\in W$ such that $b(u,w)=0$ for all $w\in \Tau_1$. Suppose that $b(u,u)<0$. Then $u\not\in \Tau_1$. Thus $\Tau_2:=\Tau_1\oplus \K u$ is a subspace of $W$ and for $w=w_1+\lambda u\in \Tau_2$, 
\[b(w,w)=b(w_1,w_1) +|\lambda|^2 b(u,u)\leq 0.\]
This contradicts the maximality of $\Tau_1$.  
  We have proved that 
$b(u,u)\geq 0$ for all $u\in \Tau_1^b$; i.e. $Z_1$ is strongly admissible.    	
\end{proof}	
Let $\Tau\subset W$ be a strongly admissible subspace and $L\in \Tau'$, where $\Tau$ carries the norm of $W$.  Then, by Theorem~\ref{th:3.5new},  (WDP) has a unique solution $u\in V$. Let $\widetilde{\Tau}\subset W$ be a maximal admissible subspace such that $\Tau\subset \widetilde{\Tau}$. By the Hahn-Banach Theorem there exists $\widetilde{L}\in\widetilde{\Tau}'$ such that $\widetilde{L}_{|\Tau}=L$. Since $\widetilde{\Tau}$ is also strongly admissible, there exists a unique $\widetilde{u}\in V$ solution for (WDP) with respect to $\widetilde{\Tau}$. 
\begin{prop}\label{prop:5.3n}
	One has $u=\widetilde{u}$. 
\end{prop}   
\begin{proof}
One has 
\[ \langle -{\mathcal D}z,u\rangle +\langle {\mathcal A} u,z\rangle =\langle L,z\rangle_{\Tau',\Tau}\mbox{ for all }z\in\Tau  \]
and 
\[ \langle -{\mathcal D}z,\widetilde{u}\rangle +\langle {\mathcal A} \widetilde{u},z\rangle =\langle L,z\rangle_{\widetilde{\Tau}',\widetilde{\Tau}}\mbox{ for all }z\in\widetilde{\Tau} . \]
In particular 
\[ \langle -{\mathcal D}z,u-\widetilde{u}\rangle +\langle {\mathcal A} (u-\widetilde{u}),z\rangle = 0\mbox{ for all }z\in\Tau . \]	
Now it follows from Theorem~\ref{th:3.5new} applied to $L=0$ that $u-\widetilde{u}=0$.   
\end{proof}	
In view of the preceding abservations, our  aim is now to study the maximal admissible subspaces in the next subsection.   
\subsection{Boundary operators}\label{sec:4}
In this section we continue to study the abstact framework from Section~\ref{sec:3}. But we give more structure to the associated boundary form $b$.  

Let $V,\r,{\mathcal D}, b$ be given as in Section~\ref{sec:3} and let $W$ be defined by (\ref{eq:3.1}) and $b:W\times W\to\K$ by (\ref{eq:3.2}). Recall that $b$ is a continuous, symmetric sesquilinear form. Now we make the following assumption. \\
\begin{assu}\label{assu:4.1}
	There are given a Hilbert space $H$ and operators 
$B_0,B_1\in{\mathcal L}(W,H)$ such that 
\begin{equation}\label{eq:4.1}
b(v,w)=\langle B_1 v,B_1w\rangle_H -\langle B_0v,B_0 w\rangle_H\mbox{ for all }v,w\in W.
\end{equation} 
and 
\begin{equation}\label{eq:4.2}
\ker B_0 + \ker B_1=W.
\end{equation}
\end{assu}
\begin{rem}\label{rem:4.2}
Such operators $B_0,B_1$ and such a Hilbert space $H$ do always exist. For example, we may choose $H=W$. There is a unique selfadjoint operator $B\in {\mathcal L}(W)$ such that 
\[ b(v,w)=\langle Bv,w\rangle _W. \]   
Then $B$ has a canonical decomposition $B=B_+-B_-$ where $B_+=\frac{1}{2} (B+|B|)$, $B_-= \frac{1}{2} (|B|-B)$ are positive selfadjoint operators, satisfying $B_+B_-=B_-B_+=0$. One may choose $B_0=\sqrt{B_+}$, $B_1=\sqrt{B_-}$, which still satisfy $B_1B_0=B_0B_1=0$ and therefore the image of one of them is included in the kernel of the other (this follows from the Spectral Theorem in \cite[Chapter VII]{RS80}).  Since $\ker B_0\oplus \overline{{\rm Im} B_0}=W$ and $\overline{{\rm Im} B_0}\subset \ker B_1$, one gets \eqref{eq:4.2}. 

Other choices of $B_0,B_1$ and $H$ different from $W$ might be more convenient and will occur in the examples in Section~\ref{sec:5n}.     
\end{rem} 
Our aim is to describe the maximal admissible subspaces in terms of $B_0$ and $B_1$. We need several auxiliary results. 

By $B_0^*, B_1^*\in {\mathcal L}(H,W)$ we denote the adjoint operators. 
\begin{lem}\label{lem:4.3}  
\begin{itemize}
	\item[a)] $\overline{ \ran B_0^*}\cap \overline{\ran B_1^*} =\{0\}$.
	\item[b)] $\r\subset \ker B_1\cap\ker B_0$. 
\end{itemize}	
\end{lem}	
\begin{proof}
a) Since by our assumption $W=\ker B_0 +\ker B_1$, it follows that 
\[
	\{0\}  =  (\ker B_0 + \ker B_1)^{\perp}= (\ker B_0)^{\perp} \cap (\ker B_1)^\perp=
	\overline{\ran B_0^*} \cap \overline{\ran B_1^*}.
\]
b) Let $r\in \r$. Then $b(r,w)=0$ for all $w\in W$. Thus 
\[  0=\langle B_1r, B_1w\rangle_H-\langle B_0 r,B_0 w\rangle_H=\langle B_1^*B_1r-B_0^*B_0r,w\rangle_W   \]
for all $w\in W$. Hence $B_1^* B_1 r=B_0^* B_0r$.  It follows from a) that 
\[  B_1^* B_1 r=B_0^* B_0 r=0.  \]
Hence $\|B_1 r\|_H^2 =\langle B_1^* B_1 r,r\rangle_W=0$ and so $B_1 r=0$. Similarly, we obtain that $B_0r=0$. 
\end{proof}	
Now we can describe a first maximal admissible subspace. 
\begin{prop}\label{prop:4.4}
$\ker B_1$ is a maximal admissible subspace. 
\end{prop}
\begin{proof}
a) $\r\subset \ker B_1$ by Lemma~\ref{lem:4.3}. \\
b) Let $w\in \ker B_1$. Then 
\[ b(w,w)=-\|  B_0 w\|_H^2\leq 0.  \]
c) Let $\Tau $ be admissible and such that $\ker B_1\subset \Tau $. Let $w\in \Tau $. By our assumptions there exists a decomposition
\[   w=w_0+w_1\in \ker B_0+  \ker B_1.\]
Since $\ker B_1\subset \Tau $ it follows that $w_0\in \Tau $.  Thus 
\[ 0\geq b(w_0,w_0)=\|B_1 w_0\|_H^2-\|B_0 w_0\|_H^2=\|B_1 w_0\|_H^2,     \]
hence $w_0\in\ker B_1$. It follows that $w=w_0+w_1\in \ker B_1$.      	
\end{proof}	
\begin{lem}\label{lem:4.5}
	Let $x_0\in\ran B_0$, $x_1\in \ran B_1$. Then there exists $w\in W$ such that $B_0 w=x_0$ and $B_1w=x_1$. 
\end{lem}
\begin{proof}
Let $x_0=B_0 w_1$, $x_1=B_1 w_2$ where $w_1,w_2\in W$.  By assumption we may write 
\begin{eqnarray*}
	w_1 & = & w_{10}+w_{11}\in \ker B_0 +\ker B_1,\\
	w_2&  = & w_{20}+w_{21}\in \ker B_0 +\ker B_1.\\ 
\end{eqnarray*}	
Let $w:=w_{11} +w_{20}$. Then $B_0 w =B_0 w_{11}=B_0 w_1=x_0$ and $B_1 w=B_1w_{20}=B_1w_2=x_1$. 
\end{proof}	
The following observation will be useful in the next proof .
\begin{rem}\label{rem:5.9n}
The operator $-{\mathcal D}_0$ is also a derivation and its extended domain is $W$. The boundary form associated with $-{\mathcal D}_0$ is $-b$ and the extended operator $-\d$. In particular 
\[   -\langle {\mathcal D}u,v\rangle_{V',V}- \overline{\langle {\mathcal D}v,u\rangle_{V',V} }=-b(u,v)=\langle B_0 u,B_0v\rangle_H - \langle B_1 u,B_1v\rangle_H,  \]
for all $u,v\in W$. 
\end{rem}

We now show that $B_0$ and $B_1$ have closed images. This is surprising and due to the special construction of $b$. For the proof we will use the well-posedness of the Derivation problem with 
respect to $\Tau =\ker B_1$. The result will be important in the sequel. 
\begin{prop}\label{prop:4.6}
$\ran B_0$ and $\ran B_1$ are closed subspaces of $H$.
\end{prop}  
\begin{proof}
a)	Let $x_0\in\overline{ \ran B_0}\subset H$. Choose $\Tau :=\ker B_1$. Then $\Tau $ is strongly admissible and $g(w)=\langle x_0,B_0 w\rangle_H$ defines a continuous antilinear form on $\Tau $, which vanishes on $\r$.  
	
Consider ${\mathcal A}\in {\mathcal L}(V,V')$ given by $\langle {\mathcal A} v,w\rangle_{V',V}=\langle v,w\rangle_V$. By Theorem~\ref{th:3.9} there exists a unique $u\in W$ such  that 
${\mathcal D} u+{\mathcal A} u=0$ and 
\[   b(u,w)=-g(w)\mbox{ for all } w\in\ker B_1 .\]	
Thus \[-\langle B_0 u,B_0w\rangle_H =-\langle x_0, B_0 w\rangle_H \mbox{ for all }w\in\ker B_1.\]
Since $\ker B_1+\ker B_0=W$, one has $B_0\ker B_1=B_0 W$. It follows that 
\[  \langle B_0 u-x_0, B_0 w\rangle_H=0\mbox{ for all }w\in W. \]
Consequently, $\langle B_0 u-x_0,y\rangle_H=0$ for all $y\in \overline{ \ran B_0} $. Since $B_0u-x_0\in   \overline{ \ran B_0}$, it follows that $B_0u-x_0=0$. Thus $x_0\in \ran  B_0$.  \\
b) Replacing ${\mathcal D}$ by $-{\mathcal D}$ and applying a), one obtains that $\ran B_1$ is 
closed.
 \end{proof}
\begin{lem}\label{lem:4.7}
Let $\Tau \subset W$ be admissible. Then there exists a linear mapping $\Phi:\ran B_0\to\ran B_1$ such that   
\[ \|\Phi x\|_H\leq \|x\|_H\mbox{ for all }x\in\ran B_0\]
 and \[\Tau \subset \Tau_{\Phi} :=\{ w\in W:B_1 w={\Phi} B_0 w\}.  \]
\end{lem}	
\begin{proof}
Since $\Tau$ is admissible,
\[    \| B_1 w\|_H\leq \| B_0 w\|_H\mbox{ for all }w\in\Tau .   \]
	Thus, letting 
	\[   {\Phi} _0 B_0 w:=B_1 w\mbox{ for all }w\in \Tau \]
		yields a well-defined mapping ${\Phi} _0:B_0\Tau \to B_1\Tau $. Then ${\Phi} _0$ is linear and contractive. Hence ${\Phi} _0$ has a unique continuous extension ${\Phi} _1:\overline{B_0\Tau }\to \overline{B_1\Tau }$ and $\|{\Phi} _1 x\|_H\leq \|x\|_H$ for all $x\in \overline{B_0 \Tau }$. Since $\ran B_0$ is closed, one has 
		\[  \ran B_0=\overline{B_0 \Tau }\oplus (B_0\Tau )^\perp.\]
Let $x=x_0+x_1\in \ran B_0$ with $x_0\in \overline{B_0\Tau }$, $x_1\in (B_0\Tau )^\perp$. Defining ${\Phi} x:={\Phi} _1x_0$ yields a linear contraction ${\Phi} :\ran B_0\to\ran B_1$ satisfying the requirements of the lemma.  			
\end{proof}	
Now we can describe all maximal admissible subspaces of $W$.  
\begin{thm}\label{th:4.8}
Le ${\Phi} :\ran B_0\to\ran B_1$ linear and contractive (i.e. $\|{\Phi} x\|_H\leq \|x\|_{H} $ for all $x\in\ran B_0$). Then 
\[  \Tau_{\Phi} :=\{ w\in W: B_1 w={\Phi} B_0 w  \} \]
is maximal admissible. Moreover, each maximal admissible subspace of $W$ is of the form $\Tau_{\Phi} $ for some linear contraction ${\Phi} :\ran B_0\to\ran B_1$.  
\end{thm}
\begin{proof}
a) Since $B_1r=B_0 r=0$ for all $r\in \r$ by Lemma~\ref{lem:4.3}, one has $\r\subset \Tau_{\Phi} $.\\
b) Let $w\in \Tau_{\Phi} $. Then 
\[b(w,w)=\|B_1 w\|^2_H-\|B_0 w\|^2_H=\|{\Phi} B_0 w\|^2_H-\|B_0w\|^2_H\leq 0.\] 	
c) Let $\Tau $ be admissible and such that $\Tau_{\Phi} \subset \Tau $.  We show that $\Tau_{\Phi} =\Tau $. In fact, by Lemma ~\ref{lem:4.7}, there exists a contraction ${\Phi} _1:\ran B_0\to\ran B_1$ such that $\Tau \subset \Tau_{{\Phi} _1}$. Hence $\Tau_{\Phi} \subset \Tau_{{\Phi} _1} $.  Thus 
\begin{equation}\label{eq:4.3}
{\Phi} B_0w={\Phi} _1B_0w\mbox{ for all }w\in \Tau_{\Phi} .
\end{equation}
But 
\begin{equation}\label{eq:4.4}
B_0 \Tau_{\Phi} =\ran B_0.
\end{equation}
In fact, let $w\in W$. By Lemma~\ref{lem:4.5} there exists $u\in W$ such that $B_0u=B_0 w$ and $B_1 u={\Phi} B_0 w$. Thus  $B_1u={\Phi} B_0u$; i.e. $u\in \Tau_{\Phi} $ and $B_0w=B_0u$. Thus $B_0 w\in B_0\Tau_{\Phi} $. 
Now (\ref{eq:4.3}) implies that ${\Phi} ={\Phi} _1$. hence $\Tau_{\Phi} \subset \Tau \subset \Tau_{{\Phi} _1}=\Tau_{\Phi} $. \\
d) The fact that each maximal subspace of $W$ is of the form $\Tau_{\Phi} $ for some linear contraction ${\Phi} : \ran B_0\to\ran B_1$ follows from Lemma~\ref{lem:4.7}.	
\end{proof}	
Next we compute the initial value condition.
\begin{prop}\label{prop:4.9}
	Let $\Phi:\ran B_0\to\ran B_1$ be a contraction. Then 
\begin{equation}\label{eq:4.5}
\Tau_{\Phi} ^b= \Tau_{{\Phi} ^*}:=\{ w\in W:B_0w={\Phi} ^*B_1w\}.
\end{equation} 
\end{prop} 
\begin{proof}
Let $u\in \Tau_{\Phi} ^b$.  Then $\langle B_1u, B_1 w\rangle_H=\langle B_0 u,B_0 w\rangle_H$ for all $w\in \Tau_{\Phi} $. Since $B_1 w={\Phi} B_0 w$ for $w\in \Tau_{\Phi} $ it follows that 
\begin{equation}\label{eq:4.6}
\langle {\Phi} ^* B_1 u-B_0 u, B_0 w\rangle_H=0\mbox{ for all }w\in \Tau_{\Phi} .
\end{equation}  
Since by (\ref{eq:4.4}) $B_0\Tau_{\Phi} =\ran B_0$ and since ${\Phi} ^* B_1 u-B_0 u\in \ran B_0$ 
 it follows from (\ref{eq:4.6}) that ${\Phi} ^*B_1 u-B_0u=0$, i.e. $B_0u={\Phi} ^*B_1u$.  	

Conversely, if $B_0u={\Phi} ^*B_1u$, then, for all $w\in \Tau_{\Phi} $, 
\begin{eqnarray*}
	b(u,w) & = & \langle B_1 u, B_1 w\rangle_H-\langle B_0 u, B_0w\rangle_H\\
	 & = & \langle B_1 u, {\Phi} B_0 w\rangle_H -\langle {\Phi} ^*B_1 u, B_0w\rangle_H\\
	  & = & 0.
\end{eqnarray*}
This proves (\ref{eq:4.5}).   
\end{proof}	
Proposition~\ref{prop:4.9} also shows that $\Tau_{\Phi} $ is strongly admissible (without using 
Proposition~\ref{prop:3.7}. In fact, by (\ref{eq:4.5}), if $b(u,w)=0$ for all $w\in \Tau_{\Phi} $, then $B_0 u={\Phi} ^* B_1 u$. Hence
\begin{eqnarray*}
b(u,u) & = & \|B_1 u\|^2_H-\|B_0u\|_H^2\\
 & = &  \|B_1 u\|^2_H-\|{\Phi} ^*B_1u\|_H^2\geq 0.
\end{eqnarray*}
\begin{lem}\label{lem:zbb}
	If $Z$ is maximal admissible, then $Z^{bb}=Z$.  
\end{lem}
\begin{proof}
	By Theorem~\ref{th:4.8}, any maximal admissible subspace of $W$  is of the form $\Tau_{\Phi} $ for some linear contraction ${\Phi} :\ran B_0\to\ran B_1$.  By Proposition~\ref{prop:4.9}, we have $\Tau_\Phi^b=\Tau_{\Phi^*}$ and therefore \[\Tau_\Phi^{bb}=\Tau_{\Phi^*}^b=\Tau_\Phi.\] 	
\end{proof}	
\subsection{Boundary conditions}
Next we compare maximal and strongly admissible subspaces. We already know that maximal  admissible subspaces are strongly admissible. Recall from Theorem~\ref{th:4.8} that the maximal admissible subspaces of $W$ are exactly the spaces $\Tau_\Phi$ where $\Phi:\ran B_0\to \ran B_1$ is a contraction.
\begin{prop}\label{prop:5.13n}
Let $\Tau\subset W$ be strongly admissible. Then there exists exactly one maximal admissible subspace $\Tau_\Phi$ such that $\Tau\subset \Tau_\Phi$. Moreover,
\begin{equation}\label{eq:5.7n}
\Tau^b =\Tau_{\Phi^\star}.
\end{equation}   
\end{prop}
\begin{proof}
Let $\Phi:\ran B_0\to \ran B_1$ be  contraction such that $\Tau\subset \Tau_\Phi$. Then 
\[   \Tau_{\Phi^*} =\Tau_\Phi^b\subset \Tau^b,  \]
where $\Tau_{\Phi^*}=\{  w\in W: B_0w=\Phi^* B_1w \}$. Note also that $-{\mathcal D}$ is a derivation with extended space $W$ and associated with $-b$. Thus, by Theorem~\ref{th:4.8}, $\Tau_{\Phi^*}$ is a maximal admissible subspace for $-b$. Since $\Tau$ is strongly admissible, the space $\Tau^b$ is admissible for $-{\mathcal D}$. Since $\Tau_{\phi^*}\subset \Tau^b$, it follows that $\Tau_{\Phi^*}=\Tau^b$. We have shown the second assertion.

Newt we show uniqueness. Let $\Psi: \ran B_0\to \ran B_1$ be another contraction such that $\Tau\subset \Tau_\Psi$. Then 
\[   \Tau_{\Psi^*}=\Tau^b=\Tau_{\Phi^*}\]
by what we have just proved. We show that $\Psi^*=\Phi^*$ (which implies that $\Psi=\Phi$). Let $w_1\in\ran B_1$, $w_0=\Phi^* w_1\in \ran B_0$. By Lemma~\ref{lem:4.5}, there exists $w\in W$ such that $w_1=B_1w$ and $w_0=B_0w$. Thus 
\[\Phi^*B_1w=\Phi^* w_1 =w_0=B_0w.\]
Hence $w\in \Tau_{\Phi^*}=\Tau_{\Psi^*}$. Consequently,
\[ \Psi^* w_1=\Psi^* B_1 w=B_0 w=w_0=\Phi^* w_1.  \]         	
\end{proof}	
An obvious  corollary is the following.  
\begin{cor}\label{cor:5.15n}
Let $\Phi,\Psi\in {\mathcal L}(H)$ be contractions. If $\Tau_\Phi=\Tau_{\Psi}$ then $\Phi=\Psi$. 
\end{cor} 

Thus, in view of Proposition~\ref{prop:5.3n}, the solutions of  (WDP) with respect to $\Tau$ and to $\Tau_\Phi$ are the same.  \\

Now in view of Theorem~\ref{th:4.8} we obtain from Theorem~\ref{th:3.5new} a theorem on existence and uniqueness of the solution to a Derivation Problem. 

Let ${\Phi} :\ran B_0\to\ran B_1$ be linear such that $\|{\Phi} \|\leq 1$. Consider the strongly admissible space 
\[   \Tau_{\Phi} =\{w\in W: B_1w={\Phi} B_0 w\}.\]
\begin{thm}\label{th:4.12}
	Let ${\Phi} :\ran B_0\to \ran B_1$ be linear and contractive for the norm of $H$. Let ${\mathcal A}\in{\mathcal L}(V,V')$ be coercive, $f\in V'$ and let $y_0\in \ran  B_0$. Then there exists a  unique $u\in W$ such that 
	 \[  {\mathcal D}u +{\mathcal A}u =f \mbox{ and }B_0 u-{\Phi} ^*B_1 u=y_0.   \]
\end{thm}
\begin{proof}
Let $Z = \Tau_\Phi$ and let us define $g\in Z'$ by $\langle g,z\rangle_{Z',Z}= \langle y_0,B_0 z\rangle_H$. Then $g$ is such  that $\langle g,r\rangle_{Z',Z}=0$ for all $r\in R$. By Theorem \ref{th:3.9}, there exists $u$ such that (SDP)  holds, that is
 \[
   u\in W , \,\, {\mathcal D}u +{\mathcal A}u =f\mbox{ and }b(u,z)=-\langle g,z\rangle_{Z',Z},\,\, z\in Z.
   \]
Since $\langle B_1 u,B_1 z\rangle_H=\langle B_1 u,\Phi B_0 z\rangle_H=\langle {\Phi}^*B_1 u,B_0 z\rangle_H$, it follows that
\[
-b(u,z)- \langle g,z\rangle_{Z',Z}= \langle B_0 u-{\Phi} ^*B_1 u -y_0,B_0 z\rangle_H=0,\,\, z\in \Tau_\Phi.
\]
This yields $B_0 u-{\Phi} ^*B_1 u -y_0 =0$, since, by Lemma \ref{lem:4.5}, $\{B_0 z, z\in \Tau_\Phi\} =\ran B_0$.

Theorem \ref{th:3.9} also shows the uniqueness of $u$ since $\Tau_\Phi$ is strongly admissible.

\end{proof}

\section{Non-autonomous evolution equations}\label{sec:appli}
In this section we establish existence and uniqueness for an evolutionary problem with very general boundary conditions concerning the time variable which can prescribe as well an initial value problem or a  periodicity condition.

Let $T>0$, ${U}$ a Hilbert space over $\K=\R$ or $\C$ which is continuous and densely embedded into another Hilbert space $H$, i.e. ${U}\stackrel{d}{\hookrightarrow} H$. As usual we identify $H$ with a subspace of ${U}'$ which yields the Gelfand triple
\[   {U}\stackrel{d}{\hookrightarrow} H\hookrightarrow {U}',\]
where for $u\in {U}$ and $ f\in H$, $\langle f,u\rangle_{{U}', {U}} =\langle f,u\rangle_H$. 

Let $V=L^2(0,T;{U})$. Thus $V'=L^2(0,T;{U}')$. Let $\r={\mathcal C}_c^\infty(0,T;{U})$ and define  ${\mathcal D}: \r\to V'$ by ${\mathcal D}r=r'$.  Then for 
$r_1,r_2\in \r$, 
\[    \langle   {\mathcal D} r_1,r_2\rangle_{V',V}  + \overline{ \langle {\mathcal D} r_2,r_1\rangle_{V',V}    }=\int_0^T ( \langle r'_1(t),r_2(t)\rangle_H  + \overline{\langle r'_2(t),r_1(t)\rangle_H   })dt =0.\] 
It is not difficult to see that $W$ as defined in Section~\ref{sec:3} is given by 
\[  W=H^1(0,T;{U}')\cap L^2(0,T;{U}).\]
The following integration by parts formula (\cite[III Corollary 1.1 p.106]{Sho97}) plays an important role. 
\begin{lem}\label{lem:5.1}  
One has $W\subset {\mathcal C}([0,T],H)$ and for all $v,w\in W,$ 
\[  \int_0^T\langle v'(t),w(t)\rangle dt =-\int_0^T  \langle w'(t),v(t)\rangle dt  + \langle v(T),w(T)\rangle_H-\langle v(0),w(0)\rangle_H .  \] 
\end{lem} 
Note that here $\langle {v'}(t),w(t)\rangle :=\langle {v'}(t),w(t)\rangle_{{U}',{U}}$. 

Thus, in this situation, the associated boundary form $b:W\times W\to\K$ is given by 
 \[ b(v,w)=   \langle v(T),w(T)\rangle_H- \langle v(0),w(0)\rangle_H.\]
 Thus it is possible to define $B_1, B_0:W\to H$  by $B_0w=w(0)$ and $B_1w=w(T)$.  Indeed, we have that for all 
 $u\in W,  u(t)=\frac{t}{T}u(t)+\frac{T-t}{T}u(t)$ with $(t\mapsto \frac{t}{T}u(t))\in \ker B_0$ and 
 $(t\mapsto\frac{T-t}{T} u(t)) \in \ker B_1$. So we have $W=\ker B_0 +\ker B_1$ where $B_0$ and $B_1$ are selfadjoint.    
 
 Now let ${\mathcal A}\in {\mathcal L}(V,V')$ be coercive. 
 \begin{exam}\label{ex:5.2}
 Let $a:[0,T]\times {U}\times {U}\to \K$ be a function such that 
 \begin{itemize}
 	\item[a)] $a(t,\cdot,\cdot):{U}\times {U}\to\K$ is sesquilinear;
 	\item[b)] $|a(t,v,w)|\leq c\|v\|_{{U} }\|w\|_{{U}}$ for all $t\in[0,T]$ and for all $v,w\in  {U}$; 
 	\item[c)] $a(\cdot, v,w):[0,T]\to\K$ is measurable for all $v,w\in {U}$;
 	\item[d)] $\re a(t,v,v)\geq \alpha \|v\|_{{U}}^2$ for all $t\in [0,T]$, $v\in {U}$ and some $\alpha>0$.   
 \end{itemize} 
For $v,w\in V=L^2(0,T;{U})$, let 
\[ \langle {\mathcal A}v,w\rangle=\int_0^Ta(t,u(t),v(t))dt. \]
Then ${\mathcal A}\in {\mathcal L}(V,V')$ is coercive. 
 \end{exam}
In the following it is important to know the trace space of $W$. 
\begin{lem}\label{lem:5.2n}
For all $x\in H$ there exists $u\in W$ such that $u(0)=x$.  
\end{lem}
\begin{proof}  
One has $H=[ V',V]_{1/2}$, the complex interpolation space, which coincides with the real interpolation space $[V',V]_{1/2,1/2}$. The latter coincides with the trace space by \cite[Proposition 1.2.10]{Lunardi}.   	
\end{proof}
Now we obtain from Theorem~\ref{th:4.12} 
the following result on existence and uniqueness. 
\begin{thm}\label{th:5.3}
Let ${\Phi} \in {\mathcal L}(H)$ be a contraction and let $y_0\in  H$. Let $f\in L^2(0,T;{U}') $. Then there exists a unique $u\in H^1(0,T;  {U}')\cap L^2(0,T;{U})$ such that 
\begin{equation}\label{eq:5.1}
u'+{\mathcal A}u=f
\end{equation} 
and 
\begin{equation}\label{eq:5.2}
u(0)-{\Phi}^* u(T)=y_0.
\end{equation}
\end{thm}
For $\Phi=0$, Theorem~\ref{th:5.3} establishes well-posedness for an initial-value problem. It is due to J. L. Lions \cite{Lio61}, see also \cite[III Proposition 2.3 p.112]{Sho97} or \cite[XVIII.3 Th\'eor\`eme 2]{DL}. 

If $\Phi=\Id$ and $y_0=0$ we obtain a periodic problem. If $\mathcal A$ is selfadjoint then Theorem~\ref{th:5.3} is given in \cite[II. Proposition 2.4]{Sho97} where selfadjointness is used in the proof on page 112. The general case given here is new and depends on our perturbation result Proposition~\ref{prop:3.7new}. \\

\noindent \textbf{Acknowledgments:}  This research is partly supported by the B\'ezout Labex, funded by ANR, reference ANR-10-LABX-58.

\end{document}